\documentclass[a4paper,12pt]{article}

\usepackage[utf8]{inputenc}
\usepackage[english]{babel}
\usepackage{amsmath, amssymb, amsthm}
\usepackage{enumerate}
\usepackage[colorlinks=true,linkcolor=blue,citecolor=blue]{hyperref}

\usepackage{tikz}
\usetikzlibrary{fadings}
\tikzfading[name=fade out,
            inner color=transparent!0,
            outer color=transparent!100]
\usetikzlibrary{calc,arrows,decorations.pathreplacing,fadings,3d,positioning}

\title{Transference theorems for Diophantine approximation with weights.
       \thanks{This research was supported by RSF grant 18-41-05001}}
\author{Oleg\,N.\,German}
\date{}

\theoremstyle{definition}
\newtheorem{definition}{Definition}

\newtheorem*{notation*}{Notation}

\theoremstyle{remark}
\newtheorem{remark}{Remark}
\newtheorem*{remark*}{Remark}

\theoremstyle{plain}
\newtheorem{theorem}{Theorem}
\newtheorem{lemma}{Lemma}

\newtheorem{proposition}{Proposition}
\newtheorem{corollary}{Corollary}

\newtheorem*{statement*}{Statement}
\newtheorem*{corollary*}{Corollary}

\DeclareMathOperator{\vol}{vol}
\DeclareMathOperator{\diam}{diam}

\DeclareMathOperator{\spanned}{span}

\renewcommand{\phi}{\varphi}

\renewcommand{\vec}[1]{\mathbf{#1}}
\renewcommand{\geq}{\geqslant}
\renewcommand{\leq}{\leqslant}

\newcommand{\R}{\mathbb{R}}
\newcommand{\Z}{\mathbb{Z}}
\newcommand{\Q}{\mathbb{Q}}

\newcommand{\La}{\Lambda}

\newcommand{\cB}{\mathcal{B}}

\newcommand{\cL}{\mathcal{L}}

\newcommand{\cP}{\mathcal{P}}
\newcommand{\cQ}{\mathcal{Q}}

\newcommand{\cS}{\mathcal{S}}

\newcommand{\tr}[1]{{#1}^\top}

\begin{document}

\maketitle

\begin{abstract}
  In this paper we prove transference inequalities for regular and uniform Diophantine exponents in the weighted setting. Our results generalise the corresponding inequalities that exist in the `non-weighted' case.
\end{abstract}


\section{Introduction}\label{sec:intro}

In 1926 A.\,Ya.\,Khintchine in his seminal paper \cite{khintchine_palermo} proved the famous transference inequalities connecting two dual problems. The first one concerns simultaneous approximation of given real numbers $\theta_1,\ldots,\theta_n$ by rationals, the second one concerns approximating zero with the values of the linear form $\theta_1x_1+\ldots+\theta_nx_n+x_{n+1}$ at integer points. Later on, Khintchine's inequalities were generalised to the case of several linear forms by F.\,Dyson \cite{dyson}. Given a matrix
\[\Theta=
  \begin{pmatrix}
    \theta_{11} & \cdots & \theta_{1m} \\
    \vdots & \ddots & \vdots \\
    \theta_{n1} & \cdots & \theta_{nm}
  \end{pmatrix}
  \in\R^{n\times m},
  \quad
  n+m\geq3,\]
let us consider the system of inequalities
\begin{equation}\label{eq:system_for_Dyson}
  \begin{cases}
    |\vec x|\leq t^{1/m} \\
    |\Theta\vec x-\vec y|\leq t^{-\gamma/n}
  \end{cases},
\end{equation}
where $\vec x\in\R^m$, $\vec y\in\R^n$, and $|\cdot|$ denotes the sup-norm.

\begin{definition}
  The \emph{Diophantine exponent} $\omega(\Theta)$ is defined as the supremum of real $\gamma$ such that the system \eqref{eq:system_for_Dyson} admits nonzero solutions in $(\vec x,\vec y)\in\Z^{m+n}$ for some arbitrarily large $t$.
\end{definition}

In this setting Dyson's result reads as follows:
\begin{equation}\label{eq:Dyson}
  \omega(\tr\Theta)\geq\frac{(n-1)+m\omega(\Theta)}{n+(m-1)\omega(\Theta)},
\end{equation}
where $\tr\Theta$ denotes the transposed matrix.

Along with the regular Diophantine exponents an important role is played by their uniform analogues.

\begin{definition}
  The \emph{uniform Diophantine exponent} $\hat\omega(\Theta)$ is defined as the supremum of real $\gamma$ such that the system \eqref{eq:system_for_Dyson} admits nonzero solutions in $(\vec x,\vec y)\in\Z^{m+n}$ for every $t$ large enough.
\end{definition}

The first transference result concerning uniform exponents belongs to V.\,Jarn\'{i}k \cite{jarnik_tiflis}. He showed that in the simplest nontrivial case $n=1$, $m=2$ we have
\begin{equation}\label{eq:Jarnik}
  \hat\omega(\tr\Theta)+\hat\omega(\Theta)^{-1}=2.
\end{equation}
In higher dimension there is no equality any longer, the corresponding inequalities
\begin{equation}\label{eq:German}
  \hat\omega(\tr\Theta)\geq
  \begin{cases}
    \dfrac{m-\hat\omega(\Theta)^{-1}}{m-1}\hskip 3mm\text{ if }\hat\omega(\Theta)\geq n/m \\
    \hskip 2mm
    \dfrac{n-1\vphantom{1^{\big|}}}{n-\hat\omega(\Theta)}\hskip 6.3mm\text{ if }\hat\omega(\Theta)\leq n/m
  \end{cases}
\end{equation}
for arbitrary $n$, $m$ were obtained by the author in \cite{german_AA_2012}, \cite{german_MJCNT_2012}.

The aim of the current paper is to prove analogues of \eqref{eq:Dyson} and \eqref{eq:German} for the so called weighted setting.

The rest of the paper is organised as follows. In Section \ref{sec:weighted_setting} we formulate our main results; in Section \ref{sec:one_linear_form} we focus on particular cases $m=1$, $n=1$ and analyse in our context a recent result by A.\,Marnat; in Section \ref{sec:inhomogeneous} we apply our results to prove transference inequalities in the inhomogeneous setting; in Sections \ref{sec:Dyson_proof}, \ref{sec:German_proof} we prove Theorems \ref{t:weighted_Dyson}, \ref{t:weighted_German}, which are the main result of the paper; and, finally, in Section \ref{sec:variety} we analyse why the generalisation of Dyson's theorem proposed in \cite{ghosh_marnat_Pisa_2019} is not optimal.

\section{Weighted setting}\label{sec:weighted_setting}

Let us fix weights $\pmb\sigma=(\sigma_1,\ldots,\sigma_m)\in\R_{>0}^m$, $\pmb\rho=(\rho_1,\ldots,\rho_n)\in\R_{>0}^n$,
\[\sigma_1\geq\ldots\geq\sigma_m,\qquad
  \rho_1\geq\ldots\geq\rho_n,\qquad
  \sum_{j=1}^m\sigma_j=\sum_{i=1}^n\rho_i=1,\]
and define the weighted norms $|\cdot|_{\pmb\sigma}$ and $|\cdot|_{\pmb\rho}$ by
\[|\vec x|_{\pmb\sigma}=\max_{1\leq j\leq m}|x_j|^{1/\sigma_j}\qquad\text{ for }\vec x=(x_1,\ldots,x_m),\]
\[|\vec y|_{\pmb\rho}=\max_{1\leq i\leq n}|y_i|^{1/\rho_i}\qquad\,\ \ \text{ for }\vec y=(y_1,\ldots,y_n).\ \]
Consider the system of inequalities
\begin{equation}\label{eq:system_with_weights}
  \begin{cases}
    |\vec x|_{\pmb\sigma}\leq t \\
    |\Theta\vec x-\vec y|_{\pmb\rho}\leq t^{-\gamma}
  \end{cases}.
\end{equation}

\begin{definition}\label{def:ordinary_weighted}
  The \emph{weighted Diophantine exponent} $\omega_{\pmb\sigma,\pmb\rho}(\Theta)$ is defined as the supremum of real $\gamma$ such that the system \eqref{eq:system_with_weights} admits nonzero solutions in $(\vec x,\vec y)\in\Z^{m+n}$ for some arbitrarily large $t$.
\end{definition}

\begin{definition}\label{def:uniform_weighted}
  The \emph{uniform weighted Diophantine exponent} $\hat\omega_{\pmb\sigma,\pmb\rho}(\Theta)$ is defined as the supremum of real $\gamma$ such that the system \eqref{eq:system_with_weights} admits nonzero solutions in $(\vec x,\vec y)\in\Z^{m+n}$ for every $t$ large enough.
\end{definition}

The following two theorems are the main result of the paper.

\begin{theorem}\label{t:weighted_Dyson}
  Set $\omega=\omega_{\pmb\sigma,\pmb\rho}(\Theta)$ and $\tr\omega=\omega_{\pmb\rho,\pmb\sigma}(\tr\Theta)$. Then
  \begin{equation}\label{eq:weighted_Dyson}
    \tr\omega\geq
    \frac{\big(\rho_n^{-1}-1\big)+\sigma_m^{-1}\omega}
         {\rho_n^{-1}+\big(\sigma_m^{-1}-1\big)\omega}.
  \end{equation}
\end{theorem}

\begin{theorem}\label{t:weighted_German}
  Set $\hat\omega=\hat\omega_{\pmb\sigma,\pmb\rho}(\Theta)$ and $\tr{\hat\omega}=\hat\omega_{\pmb\rho,\pmb\sigma}(\tr\Theta)$. Then
  \begin{equation}\label{eq:weighted_German}
    \tr{\hat\omega}\geq
    \begin{cases}
      \dfrac{1-\sigma_m\hat\omega^{-1}}{1-\sigma_m}\hskip 3mm\text{ if }\hat\omega\geq\sigma_m/\rho_n \\
      \hskip 2mm
      \dfrac{1-\rho_n\vphantom{1^{\big|}}}{1-\rho_n\hat\omega}\hskip 6.3mm\text{ if }\hat\omega\leq\sigma_m/\rho_n
    \end{cases}.
  \end{equation}
\end{theorem}

Clearly, in the `non-weighted' case, when all the $\sigma_j$ are equal to $1/m$ and all the $\rho_i$ are equal to $1/n$, \eqref{eq:weighted_Dyson} turns into \eqref{eq:Dyson}, and \eqref{eq:weighted_German} turns into \eqref{eq:German}.

We cannot avoid mentioning a recent paper \cite{ghosh_marnat_Pisa_2019} by S.\,Chow, A.\,Ghosh et al., where they propose another generalisation of Dyson's inequality, different from \eqref{eq:weighted_Dyson}. Namely, they showed that
\begin{equation}\label{eq:nedo_weighted_Dyson}
  \tr\omega\geq
  \frac{(m+n-1)\big(\rho_n^{-1}+\sigma_m^{-1}\omega\big)+\sigma_1^{-1}(\omega-1)}
       {(m+n-1)\big(\rho_n^{-1}+\sigma_m^{-1}\omega\big)-\rho_1^{-1}(\omega-1)}.
\end{equation}
For $\omega=1$ \eqref{eq:weighted_Dyson} and \eqref{eq:nedo_weighted_Dyson} obviously coincide, as well as in the Dyson's `non-weighted' case. In every other case \eqref{eq:weighted_Dyson} is strictly stronger than \eqref{eq:nedo_weighted_Dyson}. At first glance, this fact seems to be rather surprising, as both \eqref{eq:weighted_Dyson} and \eqref{eq:nedo_weighted_Dyson} are proved essentially by applying Mahler's theorem. However, there is a certain freedom of choice, at which moment to apply Mahler's theorem. Different choices result in different inequalities. We spend some time analysing this phenomenon in Section \ref{sec:variety}.

\paragraph{Concerning inversion of \eqref{eq:weighted_Dyson} and \eqref{eq:weighted_German}.}

Sometimes it is useful to have inverted versions of \eqref{eq:weighted_Dyson} and \eqref{eq:weighted_German}. For instance, such a need arises in Section \ref{sec:inhomogeneous}.

First of all let us notice that Minkowski's convex body theorem provides the following trivial bound we should always keep in mind:
\[\omega(\Theta)\geq\hat\omega(\Theta)\geq1.\]

The inequality \eqref{eq:weighted_Dyson} is very simple to invert, as the function
\[
  f(x)=\frac{\big(\rho_n^{-1}-1\big)+\sigma_m^{-1}x}
            {\rho_n^{-1}+\big(\sigma_m^{-1}-1\big)x}
\]
maps $[1,+\infty)$ monotonously onto $[1,(1-\sigma_m)^{-1})$. Furthermore, any statement concerning $\Theta$ produces a statement concerning $\tr\Theta$ by just swapping $(m,n,\pmb\sigma,\pmb\rho,\Theta)$ for $(n,m,\pmb\rho,\pmb\sigma,\tr\Theta)$. Thus, we get

\begin{corollary}\label{cor:weighted_Dyson_inversed}
  Set $\omega=\omega_{\pmb\sigma,\pmb\rho}(\Theta)$ and $\tr\omega=\omega_{\pmb\rho,\pmb\sigma}(\tr\Theta)$. Suppose $\omega<(1-\rho_n)^{-1}$.
  Then
  \begin{equation*}
    \tr\omega\leq
    \frac{\rho_n}{\sigma_m}\cdot
    \frac{\omega-(1-\sigma_m)}
         {1-(1-\rho_n)\omega}.
  \end{equation*}
\end{corollary}

As for inverting \eqref{eq:weighted_German}, it needs an additional observation due to the splitting into two cases. If $\rho_n\geq\sigma_m$, then only the first case remains, and the argument is very simple. Suppose $\rho_n<\sigma_m$. Let us set
  \[
    f_1(x)=\dfrac{1-\rho_n}{1-\rho_nx}\,,
    \qquad
    f_2(x)=\dfrac{1-\sigma_mx^{-1}}{1-\sigma_m}\,.
  \]
The mappings
\[
  f_1:
  [1,\sigma_m/\rho_n]
  \to
  \bigg[1,\frac{1-\rho_n}{1-\sigma_m}\bigg],
  \qquad
  f_2:
  [\sigma_m/\rho_n,+\infty)
  \to
  \bigg[\frac{1-\rho_n}{1-\sigma_m}\,,\frac1{1-\sigma_m}\bigg)
\]
are monotonous and one-to-one. Hence, denoting by $f_1^{-1}$ and $f_2^{-1}$ the corresponding inverse functions, we get
\begin{equation*}
  \begin{array}{l}
    \begin{cases}
      1\leq x\leq\sigma_m/\rho_n \\
      f_1(x)\leq y<\dfrac1{1-\sigma_m}
    \end{cases} \\
    \text{or} \\
    \begin{cases}
      x\geq\sigma_m/\rho_n \\
      f_2(x)\leq y<\dfrac1{1-\sigma_m}
    \end{cases}
  \end{array}
  \iff\ \ \
  \begin{array}{l}
    \begin{cases}
      1\leq y\leq\dfrac{1-\rho_n}{1-\sigma_m} \\
      1\leq x\leq f_1^{-1}(y)
    \end{cases} \\
    \text{or} \\
    \begin{cases}
      \dfrac{1-\rho_n}{1-\sigma_m}\leq y<\dfrac1{1-\sigma_m} \\
      1\leq x\leq f_2^{-1}(y)
    \end{cases}
  \end{array}
  \text{(cf. Fig.\,\ref{fig:inversion})}.
\end{equation*}

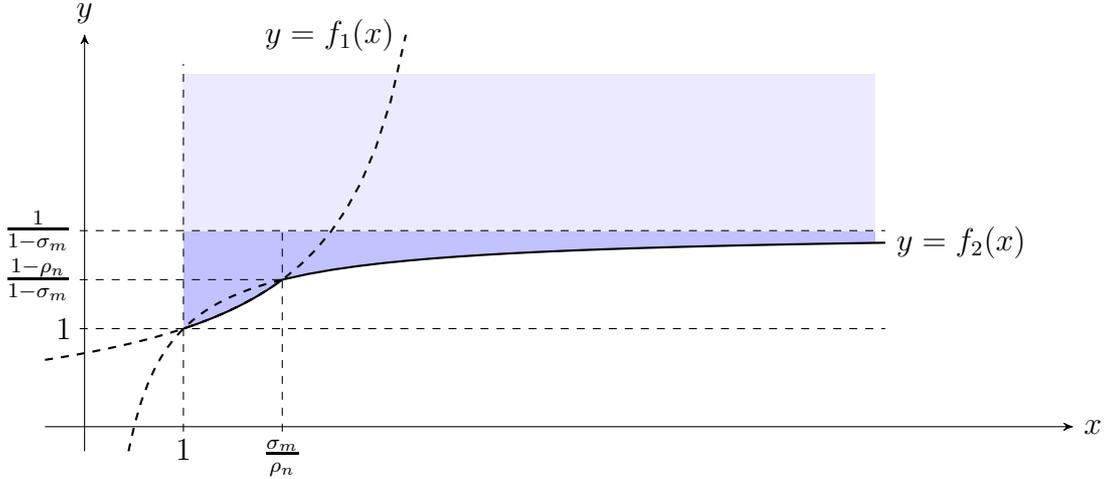
\begin{figure}[h]
\centering
\begin{tikzpicture}[scale=1.3]
  \fill[blue!24!, domain=1:2, variable=\x] plot ({\x}, {(3/4)/(1-(1/4)*\x)}) -- (2,2) -- (1,2) -- cycle;
  \fill[blue!24!, domain=2:8, variable=\x] plot (\x, {2-\x^(-1)}) -- (8,2) -- (1,2) -- cycle;
  \fill[blue!8!] (1,2) -- (8,2) -- (8,3.6) -- (1,3.6) -- cycle;

  \draw[->,>=stealth'] (-0.4,0) -- (10,0) node[right] {$x$};
  \draw[->,>=stealth'] (0,-0.25) -- (0,4) node[above] {$y$};

  \draw[color=black,dashed,thick] plot[domain=-0.4:3.25] (\x, {(3/4)/(1-(1/4)*\x)}) node[left]{$y=f_1(x)$};
  \draw[color=black,dashed,thick] plot[domain=4/9:2] (\x, {2-\x^(-1)});
  \draw[color=black,thick] plot[domain=1:2] (\x, {(3/4)/(1-(1/4)*\x)});
  \draw[color=black,thick] plot[domain=2:8.1] (\x, {2-\x^(-1)}) node[right]{$y=f_2(x)$};

  \draw[color=black,dashed] (1,-0.05) -- (1,3.7);
  \draw[color=black,dashed] (2,-0.05) -- (2,2);
  \draw[color=black,dashed] (-0.05,1) -- (8.1,1);
  \draw[color=black,dashed] (-0.05,3/2) -- (2,3/2);
  \draw[color=black,dashed] (-0.05,2) -- (8.1,2);

  \draw (1,-0.02) node[below]{$1$};
  \draw (2,-0.02) node[below]{$\frac{\sigma_m}{\rho_n}$};
  \draw (-0.02,1) node[left]{$1$};
  \draw (-0.02,3/2) node[left]{$\frac{1-\rho_n}{1-\sigma_m}$};
  \draw (-0.02,2) node[left]{$\frac1{1-\sigma_m}$};
\end{tikzpicture}
\caption{How to invert \eqref{eq:weighted_German}} \label{fig:inversion}
\end{figure}

Swapping again $(m,n,\pmb\sigma,\pmb\rho,\Theta)$ for $(n,m,\pmb\rho,\pmb\sigma,\tr\Theta)$, we get

\begin{corollary}\label{cor:weighted_German_inversed}
  Set $\hat\omega=\hat\omega_{\pmb\sigma,\pmb\rho}(\Theta)$ and $\tr{\hat\omega}=\hat\omega_{\pmb\rho,\pmb\sigma}(\tr\Theta)$. Suppose $\hat\omega<(1-\rho_n)^{-1}$.
  Then
  \begin{equation*}
    \tr{\hat\omega}\leq
    \begin{cases}
      \dfrac{\rho_n}{1-(1-\rho_n)\hat\omega}\hskip 11mm\text{ if }\hat\omega\geq\dfrac{1-\sigma_m}{1-\rho_n} \\
      \dfrac{1-(1-\sigma_m)\hat\omega^{-1}\vphantom{1^{\big|}}}{\sigma_m}\hskip 6mm\text{ if }\hat\omega\leq\dfrac{1-\sigma_m}{1-\rho_n}
    \end{cases},
  \end{equation*}
  assuming that $(1-\rho_n)^{-1}=+\infty$ for $\rho_n=1$.
\end{corollary}

\section{Case of one linear form and Marnat's examples}\label{sec:one_linear_form}

It is worth singling out the case $m=1$ and the case $n=1$, as transference theorems are more often applied in those particular cases, than in the general one.

Theorem \ref{t:weighted_Dyson} and Corollary \ref{cor:weighted_Dyson_inversed} provide the following statement for $n=1$.

\begin{theorem}\label{t:weighted_Dyson_n=1}
  Set $\omega=\omega_{\pmb\sigma,\pmb\rho}(\Theta)$ and $\tr\omega=\omega_{\pmb\rho,\pmb\sigma}(\tr\Theta)$. Suppose $n=1$. Then
  \begin{equation}\label{eq:weighted_Dyson_n=1}
    \frac{\omega}
         {\sigma_m+(1-\sigma_m)\omega}
    \leq
    \tr\omega
    \leq
    \frac{\omega-(1-\sigma_m)}
         {\sigma_m}.
  \end{equation}
\end{theorem}

As for the uniform exponents, it appears that both for $m=1$ and for $n=1$ exactly one of the inequalities \eqref{eq:weighted_German} survives. For $n=1$ we have $\rho_n=1$ and $\hat\omega(\Theta)\geq1>\sigma_m=\sigma_m/\rho_n$, i.e. the second alternative in \eqref{eq:weighted_German} is inconsistent.

The case $m=1$ is slightly more difficult. It appears that in this case $\hat\omega(\Theta)$ cannot be greater than $\rho_n^{-1}$ (unless $\Theta$ is rational), which eliminates the first alternative in \eqref{eq:weighted_German}. In fact, a stronger statement holds.

\begin{proposition}\label{prop:extra_bound}
  Let $m=1$.

  \textup{(i)} If $\Theta\in\Q^{n\times1}$, then $\hat\omega(\Theta)=\omega(\Theta)=\tr{\hat\omega(\Theta)}=\tr{\omega(\Theta)}=+\infty$.

  \textup{(ii)} If $\Theta\notin\Q^{n\times1}$, consider the minimal $k$ such that $\theta_{k1}$ is irrational. Then $\hat\omega(\Theta)\leq\rho_k^{-1}$.
\end{proposition}

\begin{proof}
  Statement \textup{(i)} is trivial. Let us prove statement \textup{(ii)}. The argument is the same as in the `non-weighted' case.

  Let $p_{\nu-1}/q_{\nu-1}$ and $p_\nu/q_\nu$ be two consecutive convergents for $\theta_{k1}$. Set $t=q_\nu-\varepsilon$ with arbitrary positive $\varepsilon$. Then for every nonzero $(x,y_1,\ldots,y_n)\in\Z^{n+1}$ such that $|x|\leq t$ we have
  \[|\theta_{k1}x-y_k|\geq|\theta_{k1}q_{\nu-1}-p_{\nu-1}|\geq\Big|\frac{p_\nu}{q_\nu}q_{\nu-1}-p_{\nu-1}\Big|=\frac1{q_\nu}=(t+\varepsilon)^{-1}.\]
  Thus, given $\gamma>\rho_k^{-1}$, one can find $t$, arbitrarily large, for which the system \eqref{eq:system_with_weights} admits no nonzero integer solutions. Hence $\hat\omega(\Theta)\leq\rho_k^{-1}$.
\end{proof}

Theorem \ref{t:weighted_German}, Corollary \ref{cor:weighted_German_inversed}, and Proposition \ref{prop:extra_bound} provide the following statement for $n=1$.

\begin{theorem}\label{t:weighted_German_n=1}
  Set $\hat\omega=\hat\omega_{\pmb\sigma,\pmb\rho}(\Theta)$ and $\tr{\hat\omega}=\hat\omega_{\pmb\rho,\pmb\sigma}(\tr\Theta)$. Suppose $n=1$ and $\Theta\notin\Q^{1\times m}$. Then
  \begin{equation}\label{eq:weighted_German_n=1}
    \begin{aligned}
      & (1-\sigma_m)\tr{\hat\omega}+\sigma_m\hat\omega^{-1}\geq1, \\
      & \sigma_m\tr{\hat\omega}+(1-\sigma_m)\hat\omega^{-1}\leq1.
      \vphantom{1^{|}}
    \end{aligned}
  \end{equation}
  Moreover, we also have
  \begin{equation*} 
    \sigma_k\tr{\hat\omega}\leq1,
  \end{equation*}
  where $k$ is the minimal index such that $\theta_{1k}$ is irrational.
\end{theorem}

For $m=2$ we obviously have $\sigma_m=\sigma_2$ and $1-\sigma_m=\sigma_1$, which makes \eqref{eq:weighted_German_n=1} look even nicer.

\begin{corollary}\label{cor:weighted_German_n=1_m=2}
  Set $\hat\omega=\hat\omega_{\pmb\sigma,\pmb\rho}(\Theta)$ and $\tr{\hat\omega}=\hat\omega_{\pmb\rho,\pmb\sigma}(\tr\Theta)$. Suppose $n=1$, $m=2$, and $\Theta\notin\Q^{1\times2}$. Then
  \begin{equation}\label{eq:weighted_German_n=1_m=2}
    \begin{aligned}
      & \sigma_1\tr{\hat\omega}+\sigma_2\hat\omega^{-1}\geq1, \\
      & \sigma_2\tr{\hat\omega}+\sigma_1\hat\omega^{-1}\leq1.
      \vphantom{1^{|}}
    \end{aligned}
  \end{equation}
  Moreover, if $\theta_{11}$ is irrational, we also have
  \begin{equation}\label{eq:weighted_German_n=1_m=2_extra}
    \sigma_1\tr{\hat\omega}\leq1.
  \end{equation}
\end{corollary}

\begin{remark}
  If, within the hypothesis of Corollary \ref{cor:weighted_German_n=1_m=2}, $\theta_{11}$ is rational, then $\hat\omega=+\infty$, whereas for $\tr\Theta$ the system \eqref{eq:system_with_weights} reduces to a system
  \[\begin{cases}
      |x|\leq t \\
      |\theta x-y|\leq t^{-\gamma\sigma_2}
  \end{cases}\]
  with an irrational $\theta$, which by the argument in the spirit of Proposition \ref{prop:extra_bound} implies that $\tr{\hat\omega}=\sigma_2^{-1}$. Thus, in this case we always have $(\hat\omega,\tr{\hat\omega})=(+\infty,\sigma_2^{-1})$.
\end{remark}

It is very interesting now to analyse a result by A.\,Marnat \cite{marnat_twisted}, who proved the existence of uncountably many $\Theta$ for $n=1$, $m=2$ with prescribed values of $\hat\omega$ and $\tr{\hat\omega}$, showing thus that there is no analogue of Jarn\'{i}k's relation in the weighted case. Namely, he proved that for every positive $a<(3\sigma_1)^{-1}$ and every $b$ satisfying the inequalities
\[\begin{aligned}
    & \sigma_1b+\sigma_2a\geq1, \\
    & \sigma_2b+\sigma_1a\leq1, \\
    & \sigma_1b\leq1,
  \end{aligned}\]
there exist uncountably many $\Theta$ with $\hat\omega=a^{-1}$ and $\tr{\hat\omega}=b$.

Particularly, in the case of irrational $\theta_{11}$, it follows from Marnat's result that for $\hat\omega>3\sigma_1$ the inequalities \eqref{eq:weighted_German_n=1_m=2}, \eqref{eq:weighted_German_n=1_m=2_extra} are sharp. Of course, it would be interesting to prove this fact for every $\hat\omega\geq1$.

\section{Application to inhomogeneous approximation}\label{sec:inhomogeneous}

Another important class of Diophantine problems concerns the inhomogeneous setting. Given $\pmb\eta=(\eta_1,\ldots,\eta_n)\in\R^n$, consider the system
\begin{equation}\label{eq:inhomogeneous_system_with_weights}
  \begin{cases}
    |\vec x|_{\pmb\sigma}\leq t \\
    |\Theta\vec x-\vec y-\pmb\eta|_{\pmb\rho}\leq t^{-\gamma}
  \end{cases}.
\end{equation}

\begin{definition}
  The \emph{inhomogeneous weighted Diophantine exponent} $\omega_{\pmb\sigma,\pmb\rho}(\Theta,\pmb\eta)$ is defined as the supremum of real $\gamma$ such that the system \eqref{eq:inhomogeneous_system_with_weights} admits nonzero solutions in $(\vec x,\vec y)\in\Z^{m+n}$ for some arbitrarily large $t$.
\end{definition}

\begin{definition}
  The \emph{inhomogeneous uniform weighted Diophantine exponent} $\hat\omega_{\pmb\sigma,\pmb\rho}(\Theta,\pmb\eta)$ is defined as the supremum of real $\gamma$ such that the system \eqref{eq:inhomogeneous_system_with_weights} admits nonzero solutions in $(\vec x,\vec y)\in\Z^{m+n}$ for every $t$ large enough.
\end{definition}

In the aforementioned paper \cite{ghosh_marnat_Pisa_2019} S.\,Chow, A.\,Ghosh et al. proved the following inequalities, the `non-weighted' version of which belongs to M.\,Laurent and Y.\,Bugeaud \cite{bugeaud_laurent_2005}:
\begin{equation}\label{eq:weighted_Bugeaud_Laurent}
  \omega_{\pmb\sigma,\pmb\rho}(\Theta,\pmb\eta)\geq
  \frac1{\hat\omega_{\pmb\rho,\pmb\sigma}(\tr\Theta)},\qquad
  \hat\omega_{\pmb\sigma,\pmb\rho}(\Theta,\pmb\eta)\geq
  \frac1{\omega_{\pmb\rho,\pmb\sigma}(\tr\Theta)}.
\end{equation}
Corollaries \ref{cor:weighted_Dyson_inversed} and \ref{cor:weighted_German_inversed} combined with \eqref{eq:weighted_Bugeaud_Laurent} provide the following two results. A similar approach was used in \cite{beresnevich_velani_London_2010} and \cite{ghosh_marnat_Cambridge_2019} in the `non-weighted' case.

\begin{theorem}\label{t:weighted_inhomogeneous_Dyson}
  Set $\omega=\omega_{\pmb\sigma,\pmb\rho}(\Theta)$ and $\hat\omega_{\pmb\eta}=\hat\omega_{\pmb\sigma,\pmb\rho}(\Theta,\pmb\eta)$. Suppose $\omega<(1-\rho_n)^{-1}$.
  Then
  \begin{equation}\label{eq:weighted_inhomogeneous_Dyson}
    \hat\omega_{\pmb\eta}\geq
    \frac{\sigma_m}{\rho_n}\cdot
    \frac{1-(1-\rho_n)\omega}
         {\omega-(1-\sigma_m)}.
  \end{equation}
\end{theorem}

\begin{theorem}\label{t:weighted_inhomogeneous_German}
  Set $\hat\omega=\hat\omega_{\pmb\sigma,\pmb\rho}(\Theta)$ and $\omega_{\pmb\eta}=\omega_{\pmb\sigma,\pmb\rho}(\Theta,\pmb\eta)$. Suppose $\hat\omega<(1-\rho_n)^{-1}$. Then
  \begin{equation}\label{eq:weighted_inhomogeneous_German}
    \omega_{\pmb\eta}\geq
    \begin{cases}
      \dfrac{1-(1-\rho_n)\hat\omega}{\rho_n}\hskip 11mm\text{ if }\hat\omega\geq\dfrac{1-\sigma_m}{1-\rho_n} \\
      \dfrac{\sigma_m\vphantom{1^{\big|}}}{1-(1-\sigma_m)\hat\omega^{-1}}\hskip 6mm\text{ if }\hat\omega\leq\dfrac{1-\sigma_m}{1-\rho_n}
    \end{cases},
  \end{equation}
  assuming that $(1-\rho_n)^{-1}=+\infty$ for $\rho_n=1$.
\end{theorem}

Notice that due to the trivial inequalities $\omega_{\pmb\eta}\geq\hat\omega_{\pmb\eta}$, $\omega\geq\hat\omega$ both \eqref{eq:weighted_inhomogeneous_Dyson} and \eqref{eq:weighted_inhomogeneous_German} provide lower estimates for $\omega_{\pmb\eta}$ in terms of $\omega$. One can easily check that the one provided by \eqref{eq:weighted_inhomogeneous_Dyson} is weaker than the one provided by \eqref{eq:weighted_inhomogeneous_German}. However, there is a small disadvantage in the latter caused by the condition on $\hat\omega$. But in the cases $m=1$ and $n=1$ that condition luckily disappears, which turns Theorems \ref{t:weighted_inhomogeneous_Dyson}, \ref{t:weighted_inhomogeneous_German} into the following symmetric statement.

\begin{theorem}\label{t:weighted_inhomogeneous_n=1_or_m=1}
  Let $\omega$, $\hat\omega$, $\omega_{\pmb\eta}$, $\hat\omega_{\pmb\eta}$ be as in Theorems \ref{t:weighted_inhomogeneous_Dyson}, \ref{t:weighted_inhomogeneous_German}.

  \textup{(i)} Suppose $n=1$. Then
  \begin{equation*}
    \hat\omega_{\pmb\eta}\geq
    \frac{\sigma_m}{\omega-(1-\sigma_m)},
    \qquad
    \omega_{\pmb\eta}\geq
    \frac{\sigma_m}{1-(1-\sigma_m)\hat\omega^{-1}}.
  \end{equation*}

  \textup{(ii)} Suppose $m=1$ and $\omega<(1-\rho_n)^{-1}$. Then
  \begin{equation*}
    \hat\omega_{\pmb\eta}\geq
    \frac{\omega^{-1}-(1-\rho_n)}{\rho_n},
    \qquad
    \omega_{\pmb\eta}\geq
    \frac{1-(1-\rho_n)\hat\omega}{\rho_n}.\phantom{11}
  \end{equation*}
\end{theorem}

\section{Dyson's transference with weights}\label{sec:Dyson_proof}

In this Section we prove Theorem \ref{t:weighted_Dyson}.

\subsection{Mahler's theorem in terms of pseudo-compound paralle\-le\-pi\-peds}\label{sec:Mahler_by_German}

All the Dyson-like transference theorems base upon a phenomenon described in its utmost generality by the classical Mahler theorem on a bilinear form (see \cite{mahler_casopis_linear}, \cite{mahler_casopis_convex}, \cite{cassels_DA}). We believe that from the geometric point of view Mahler's theorem is most vividly formulated in terms of pseudo-compound parallelepipeds and dual lattices. An interested reader can find this interpretation performed in detail in \cite{german_evdokimov} (see also \cite{schmidt_DA} for more information about pseudo-compounds in the context of Mahler's theory). In this Section we simply formulate the corresponding version of Mahler's theorem (Theorem \ref{t:Mahler_by_German} below).

\begin{definition}\label{def:pseudo_compound}
  Given positive numbers $\lambda_1,\ldots,\lambda_d$, consider the parallelepiped
  \[\cP=\Big\{ \vec z=(z_1,\ldots,z_d)\in\R^d \,\Big|\, |z_i|\leq\lambda_i,\ i=1,\ldots,d \Big\}.\]
  We call the parallelepiped
  \[\cP^\ast=\Big\{ \vec z=(z_1,\ldots,z_d)\in\R^d \,\Big|\, |z_i|\leq\frac1{\lambda_i}\prod_{j=1}^d\lambda_j,\ i=1,\ldots,d \Big\}\]
  the \emph{pseudo-compound} of $\cP$.
\end{definition}

We remind that, given a full-rank lattice $\La$ in $\R^d$, its \emph{dual} is defined as
\[\La^\ast=\big\{\, \vec z\in\R^d \,\big|\ \langle\vec z,\vec w\rangle\in\Z\text{ for all }\vec w\in\La \,\big\},\]
where $\langle \,\cdot\,,\cdot\,\rangle$ denotes inner product.

\begin{theorem}[Interpretation of Mahler's theorem] \label{t:Mahler_by_German}
  Let $\cP$ be as in Definition \ref{def:pseudo_compound}. Let $\La$ be a full-rank lattice in $\R^d$, $d\geq2$, $\det\La=1$. Then
  \[\cP^\ast\cap\La^\ast\neq\{\vec 0\}
    \implies
    c\cP\cap\La\neq\{\vec 0\},\]
  where $c=d^{\raisebox{1ex}{$\frac1{2(d-1)}$}}$ and $\vec 0$ denotes the origin.
\end{theorem}

With the given value of $c$ Theorem \ref{t:Mahler_by_German} was proved in \cite{german_evdokimov}. In Mahler's formulation $c$ equals $d-1$. However, for our purposes any constant depending only on $d$ will do, as we are concerned only with exponents.

\subsection{Dual lattices and two-parametric families of parallelepipeds}\label{sec:dual_lattices_and_parallelepipeds}

Set $d=m+n$. Then, as assumed in the beginning of the paper, $d\geq3$. Define
\[\La=
  \begin{pmatrix}
    \vec I_m & \\
    -\Theta  & \vec I_n
  \end{pmatrix}
  \Z^d.\]
Then the dual lattice is given by
\[\La^\ast=
  \begin{pmatrix}
    \vec I_m & \tr\Theta \\
             & \vec I_n
  \end{pmatrix}
  \Z^d.\]
For each $t>1$, $\gamma\geq1$, $s>1$, $\delta\geq1$ set
\begin{equation*}
\begin{aligned}
  \cP(t,\gamma) & =\Bigg\{\,\vec z=(z_1,\ldots,z_d)\in\R^d \ \Bigg|
                        \begin{array}{l}
                          |(z_1,\ldots,z_m)|_{\pmb\sigma}\leq t \\
                          |(z_{m+1},\ldots,z_d)|_{\pmb\rho}\leq t^{-\gamma}
                        \end{array} \Bigg\}, \\
  \cQ(s,\delta) & =\Bigg\{\,\vec z=(z_1,\ldots,z_d)\in\R^d \ \Bigg|
                        \begin{array}{l}
                          |(z_1,\ldots,z_m)|_{\pmb\sigma}\leq s^{-\delta} \\
                          |(z_{m+1},\ldots,z_d)|_{\pmb\rho}\leq s
                        \end{array} \Bigg\}.
\end{aligned}
\end{equation*}
We can reformulate Definition \ref{def:ordinary_weighted} in the following way.

\begin{proposition}\label{prop:regular_omega_def_reformulation}
  \begin{equation*}
  \begin{aligned}
    \omega_{\pmb\sigma,\pmb\rho}(\Theta) & =\sup\Big\{ \gamma\geq1 \,\Big|\, \text{there is $t$, however large, s.t. }\cP(t,\gamma)\cap\La\neq\{\vec 0\} \Big\}, \\
    \omega_{\pmb\rho,\pmb\sigma}(\tr\Theta) & =\sup\Big\{ \delta\geq1 \,\Big|\, \text{there is $s$, however large, s.t. }\cQ(s,\delta)\cap\La^\ast\neq\{\vec 0\} \Big\}.
  \end{aligned}
  \end{equation*}
\end{proposition}

Now, the preparations are complete, and we are ready to prove Theorem \ref{t:weighted_Dyson}. 

\subsection{Proof of Theorem \ref{t:weighted_Dyson}}\label{sec:twisted_Dyson_proof_itself}

For every $t,\gamma\in\R$ such that $t>1$, $1\leq\gamma<(1-\rho_n)^{-1}$ set
\begin{equation}\label{eq:s_and_delta_via_t_and_gamma}
  s=t^{\gamma-(\gamma-1)\rho_n^{-1}},\qquad
  \delta=\frac{1+(\gamma-1)\sigma_m^{-1}}{\gamma-(\gamma-1)\rho_n^{-1}}.
\end{equation}
Then
\[\begin{aligned}
    s^{-\delta\sigma_j} & \leq t^{-\sigma_j+1-\gamma},\qquad j=1,\ldots,m, \\
    s^{\rho_i} & \leq t^{\gamma\rho_i+1-\gamma},\qquad\,\ i=1,\ldots,n,
  \end{aligned}\]
whence
\begin{equation}\label{eq:Q_subset_P}
  \cQ(s,\delta)\subseteq
  \cP(t,\gamma)^\ast\ ,
\end{equation}
\begin{equation*}
  \cP(t,\gamma)^\ast=\Bigg\{\,\vec z=(z_1,\ldots,z_d)\in\R^d \ \Bigg|
                             \begin{array}{lr}
                               |z_j|\leq t^{-\sigma_j+1-\gamma}, & j=1,\ldots,m \\
                               |z_{m+i}|\leq t^{\gamma\rho_i+1-\gamma}, & i=1,\ldots,n
                             \end{array} \Bigg\}.
\end{equation*}
Combining \eqref{eq:Q_subset_P} with Theorem \ref{t:Mahler_by_German}, we get the key relation
\begin{equation}\label{eq:if_Q_is_nonempty_then_so_is_P}
  \cQ(s,\delta)\cap\La^\ast\neq\{\vec 0\}
  \implies
  c\cP(t,\gamma)\cap\La\neq\{\vec 0\}.
\end{equation}

The assumption $1\leq\gamma<(1-\rho_n)^{-1}$ guarantees that the correspondence $\gamma\mapsto\delta$ given by \eqref{eq:s_and_delta_via_t_and_gamma} generates a one-to-one monotonous mapping $[1,(1-\rho_n)^{-1})\to[1,+\infty)$. Particularly, $s$ and $t$ tend to $+\infty$ simultaneously, and $\gamma$ can be expressed as a function of $\delta$.

Thus, in view of Proposition \ref{prop:regular_omega_def_reformulation}, \eqref{eq:if_Q_is_nonempty_then_so_is_P} implies that
\begin{equation}\label{eq:geq_delta_implies_geq_gamma}
  \omega_{\pmb\rho,\pmb\sigma}(\tr\Theta)\geq\delta
  \implies
  \omega_{\pmb\sigma,\pmb\rho}(\Theta)\geq\gamma=
  \frac{\big(\sigma_m^{-1}-1\big)+\rho_n^{-1}\delta}
       {\sigma_m^{-1}+\big(\rho_n^{-1}-1\big)\delta}.
\end{equation}
Hence
\[\omega_{\pmb\sigma,\pmb\rho}(\Theta)\geq
  \frac{\big(\sigma_m^{-1}-1\big)+\rho_n^{-1}\omega_{\pmb\rho,\pmb\sigma}(\tr\Theta)}
       {\sigma_m^{-1}+\big(\rho_n^{-1}-1\big)\omega_{\pmb\rho,\pmb\sigma}(\tr\Theta)}.\]
Swapping $(\pmb\sigma,\pmb\rho,\Theta)$ for $(\pmb\rho,\pmb\sigma,\tr\Theta)$ gives \eqref{eq:weighted_Dyson}. Theorem \ref{t:weighted_Dyson} is proved.

It is clear that \eqref{eq:if_Q_is_nonempty_then_so_is_P} also provides an analogue of \eqref{eq:weighted_Dyson} for uniform exponents, but there is no need for such an analogue, as we are about to prove a stronger statement, namely, Theorem \ref{t:weighted_German}.

\section{Uniform transference with weights}\label{sec:German_proof}

In this Section we prove Theorem \ref{t:weighted_German}.

\subsection{Analogue of Theorem \ref{t:Mahler_by_German} for second pseudo-compounds}

As we noticed in the beginning of Section \ref{sec:Mahler_by_German}, Theorem \ref{t:Mahler_by_German} is the core of any transference theorem for regular exponents. But if we want to prove something about uniform exponents, we must use a more delicate tool. In this Section we propose an analogue of Theorem \ref{t:Mahler_by_German} dealing with pairs of lattice points (Theorem \ref{t:second_pseudo_compound} below). The idea of this approach was used by the author in \cite{german_MJCNT_2012} to prove \eqref{eq:German}.

Let $\vec e_1,\ldots,\vec e_d$ be the standard basis of $\R^d$. Let us associate to each $\vec Z\in\bigwedge^2\R^d$ its representation
\[\vec Z=\sum_{1\leq i<j\leq d}Z_{ij}\vec e_i\wedge\vec e_j.\]

\begin{definition}\label{def:second_pseudo_compound}
  Given positive numbers $\lambda_1,\ldots,\lambda_d$, consider the parallelepiped
  \[\cP=\Big\{ \vec z=(z_1,\ldots,z_d)\in\R^d \,\Big|\, |z_i|\leq\lambda_i,\ i=1,\ldots,d \Big\}.\]
  We call the parallelepiped
  \[\cP^\circledast=\bigg\{ \vec Z 
                                   \in{\textstyle\bigwedge^2\R^d} \,\bigg|\,
                             |Z_{ij}|\leq\frac1{\lambda_i\lambda_j}\prod_{k=1}^d\lambda_k,\ 1\leq i<j\leq d \bigg\}\]
  the \emph{second pseudo-compound} of $\cP$.
\end{definition}

\begin{remark}
  Our terminology differs a bit from that which W.\,M.\,Schmidt uses in his exposition of Mahler's theory in \cite{schmidt_DA}. Instead of $\cP^\ast$ and $\cP^\circledast$ he actually considers $\star\cP^\ast$ and $\star\cP^\circledast$ -- the images of $\cP^\ast$ and $\cP^\circledast$ under the action of the Hodge star operator. Respectively, he calls them the $(d-1)$-th and the $(d-2)$-th pseudo-compounds of $\cP$. It agrees well with Mahler's definition of compound bodies \cite{mahler_compound_bodies_I}, \cite{mahler_compound_bodies_II}, but in our context
  it seems more appropriate to omit the Hodge star and reverse the numeration order.
\end{remark}

Given a full-rank lattice $\La$ in $\R^d$ and its dual $\La^\ast$, let us denote by $\La^{\circledast}$ the set of decomposable elements of the lattice $\bigwedge^2\La^\ast$, i.e.
\begin{equation*}
  \La^{\circledast}=\Big\{ \vec z_1\wedge\vec z_2 \,\Big|\, \vec z_1,\vec z_2\in\La^\ast \Big\}.
\end{equation*}

\begin{theorem}\label{t:second_pseudo_compound}
  Let $\cP$ be as in Definition \ref{def:second_pseudo_compound}. Let $\La$ be a full-rank lattice in $\R^d$, $\det\La=1$. Then
  \[\cP^{\circledast}\cap\La^{\circledast}\neq\{\vec 0\}
    \implies
    c'\cP\cap\La\neq\{\vec 0\},\]
  where $c'=\big(\frac{d(d-1)}2\big)^{\raisebox{1ex}{$\frac1{2(d-2)}$}}$ and $\vec 0$ denotes the origin.
\end{theorem}

Proof of Theorem \ref{t:second_pseudo_compound} is based on three facts. The first one is Minkowski's convex body theorem, the second one is Vaaler's theorem \cite{vaaler}, which states that the $k$-dimensional volume of any $k$-dimensional central section of a unit cube is not less than $1$, and the third one is the following observation.

\begin{proposition}\label{prop:orthogonal_sublattices}
  Let $\La$ be a full-rank lattice in $\R^d$, $\det\La=1$. Let $\cL$ be a $k$-dimensional subspace of $\R^d$. Suppose $\Gamma=\cL\cap\La$ has rank $k$. Consider the orthogonal complement $\cL^\perp$ and denote $\Gamma^\perp=\cL^\perp\cap\La^\ast$. Then $\Gamma^\perp$ has rank $d-k$ and
  \[\det\Gamma^\perp=\det\Gamma.\]
\end{proposition}

Since the lattice is assumed to be unimodular, Proposition \ref{prop:orthogonal_sublattices} by linearity reduces to the case $\La=\Z^d$, which seems to be a rather classical statement. The corresponding proof can be found at least in \cite{schmidt_DADE} and \cite{german_MJCNT_2012}.

\begin{proof}[Proof of Theorem \ref{t:second_pseudo_compound}]
  Set
  \[v=\frac12(\vol\cP)^{1/d}=\prod_{k=1}^d\lambda_k^{1/d}.\]
  Consider the diagonal matrix $T=\textup{diag}(\lambda_1/v,\ldots,\lambda_d/v)$. Then $T^{-1}\cP=v\cB$, where
  \[\cB=\Big\{ \vec z=(z_1,\ldots,z_d)\in\R^d \,\Big|\, |z_i|\leq1,\ i=1,\ldots,d \Big\}.\]
  As $(T^{-1}(\La))^\ast=\tr T(\La^\ast)=T(\La^\ast)$, the second compound matrix $T^{(2)}$ is acting on $\bigwedge^2\R^d$ thought of as the ambient space for $\bigwedge^2\La^\ast$.
%
%
  Since
  \[T^{(2)}(\cP^\circledast)=v^{d-2}\cB^\circledast,\]
  where
  \[\cB^\circledast=\bigg\{ \vec Z 
                                   \in{\textstyle\bigwedge^2\R^d} \,\bigg|\,
                             |Z_{ij}|\leq1,\ 1\leq i<j\leq d \bigg\},\]
  we are to show that
  \begin{equation}\label{eq:second_pseudo_compound_for_cubes}
    (v^{d-2}\cB^\circledast)\cap(T^{-1}\La)^\circledast\neq\{\vec 0\}
    \implies
    (c'v\cB)\cap(T^{-1}\La)\neq\{\vec 0\}.
  \end{equation}
  Now, the left hand side of \eqref{eq:second_pseudo_compound_for_cubes} implies that there is a sublattice $\Gamma$ in $(T^{-1}\La)^\ast$ of rank $2$ with
  \[\det\Gamma\leq v^{d-2}\frac{\diam\cB^\circledast}2=v^{d-2}\bigg(\frac{d(d-1)}2\bigg)^{1/2}=(c'v)^{d-2}.\]
  The determinant of $(T^{-1}\La)^\ast$ equals $1$, so by Proposition \ref{prop:orthogonal_sublattices} there is a sublattice $\Gamma^\perp$ in $T^{-1}\La$ of rank $d-2$ with
  \[\det\Gamma^\perp=\det\Gamma\leq(c'v)^{d-2}.\]
  By Vaaler's theorem the $(d-2)$-dimensional volume of $\cS=\spanned_\R(\Gamma^\perp)\cap(c'v\cB)$ is not less than $(2c'v)^{d-2}$. Applying Minkowski's convex body theorem, we get that there is a nonzero point of $\Gamma^\perp$ in $\cS$, which completes the proof.
\end{proof}

\subsection{`Nodes' and `leaves': main parametric construction}\label{sec:nodes_and_leaves}

Let us adopt the notation of Section \ref{sec:dual_lattices_and_parallelepipeds}. Our proof of Theorem \ref{t:weighted_German} bases upon a construction involving parallelepipeds $\cQ(\,\cdot\ ,\,\cdot\,)$ defined in Section \ref{sec:dual_lattices_and_parallelepipeds}. In this Section we describe this construction.

Let us fix arbitrary $s,\delta,\alpha\in\R$ such that
\[s>1,\quad\delta\geq1,\quad1\leq\alpha\leq\delta,\]
and denote
\[S=s^{\delta/\alpha}.\]
To each $r>1$ let us associate the parallelepiped
\begin{equation*}
  \cQ_r =\cQ\big(r,\alpha\tfrac{\log(sS/r)}{\log r}\big)
        =\Bigg\{\,\vec z 
                  \in\R^d \ \Bigg|
                  \begin{array}{l}
                    |(z_1,\ldots,z_m)|_{\pmb\sigma}\leq (sS/r)^{-\alpha} \\
                    |(z_{m+1},\ldots,z_d)|_{\pmb\rho}\leq r
                  \end{array} \Bigg\}.
\end{equation*}
Consider the following three families of parallelepipeds:
\begin{equation*}
\begin{aligned}
  \mathfrak S & =\mathfrak S(s,\delta,\alpha)=\Big\{ \cQ_r \,\Big|\, s\leq r\leq\sqrt{sS} \Big\}, \\
  \mathfrak A & =\mathfrak A(s,\delta,\alpha)=\Big\{ \cQ_r \,\Big|\, \sqrt{sS}\leq r\leq S \Big\}, \\
  \mathfrak L & =\mathfrak L(s,\delta,\alpha)=\Big\{ \cQ(r,\alpha) \,\Big|\, s\leq r\leq S \Big\}.
\end{aligned}
\end{equation*}
Let us call $\mathfrak S$ the \emph{`stem' family}, $\mathfrak A$ the \emph{`anti-stem' family}, $\mathfrak L$ the \emph{`leaves' family}. Let us call each element of $\mathfrak S$ a \emph{`node'}, each element of $\mathfrak A$ an \emph{`anti-node'}, each element of $\mathfrak L$ a \emph{`leaf'}.

We say that a `node' or an `anti-node' $\cQ_r$ \emph{produces} a `leaf' $\cQ(r',\alpha)$ if
\begin{equation*}
  r'=r\quad\text{ or }\quad r'=sS/r.
\end{equation*}

\begin{proposition}\label{prop:leaves_production}
  \indent

  \textup{(i)}
  If $r<r'$, then $\cQ_r\subset\cQ_{r'}$. For the root `node' $\cQ_s$ we have $\cQ_s=\cQ(s,\delta)$.

  \textup{(ii)}
  For each $r\in\R$, $s\leq r\leq\sqrt{sS}$, the `node' $\cQ_r$ and the `anti-node' $\cQ_{sS/r}$ produce exactly two `leaves'
  \[\cQ(r,\alpha)
    \quad\text{ and }\quad
    \cQ(sS/r,\alpha),\]
  whose intersection is the `node' and whose union is contained in the `anti-node'.

  \textup{(iii)}
  Each `leaf' $\cQ(r,\alpha)$ is produced by exactly one `node' $\cQ_{r'}$ and one `anti-node' $\cQ_{sS/r'}$ with
  \begin{equation*}
    r'=
    \begin{cases}
      r,\qquad\ \,\text{ if }r\leq\sqrt{sS} \\
      sS/r,  \quad\text{ if }r\geq\sqrt{sS}
    \end{cases}.
  \end{equation*}
\end{proposition}

\begin{proof}
  All three statements follow directly from the definition of $\cQ_r$ and the definition of \emph{producing}.
\end{proof}

We illustrate Proposition \ref{prop:leaves_production} by Figure \ref{fig:nodes_and_leaves}, where we use $u$ and $v$ to denote $|(z_{m+1},\ldots,z_d)|_{\pmb\rho}$ and $|(z_1,\ldots,z_m)|_{\pmb\sigma}$ respectively.

\begin{figure}[h]
\centering
\begin{tikzpicture}[scale=5]
  \fill[blue!10!] (0,0.5*1.5^1.5) -- (1.5,0.5*1.5^1.5) -- (1.5,0) -- (0,0) -- cycle;
  \fill[blue!20!] (0,0.5*2^1.5/3^1.5) -- (1.5,0.5*2^1.5/3^1.5) -- (1.5,0) -- (0,0) -- cycle;
  \fill[blue!20!] (0,0.5*1.5^1.5) -- (2/3,0.5*1.5^1.5) -- (2/3,0) -- (0,0) -- cycle;
  \fill[blue!30!] (0,0.5*2^1.5/3^1.5) -- (2/3,0.5*2^1.5/3^1.5) -- (2/3,0) -- (0,0) -- cycle;

  \draw[->,>=stealth'] (-0.2,0) -- (2.6,0) node[right] {$u$};
  \draw[->,>=stealth'] (0,-0.15) -- (0,0.5*3.4) node[above] {$v$};

  \draw[color=black] plot[domain=0:2.2] (\x, 0.5*\x^1.5) node[above right]{$v=\big(sS/u\big)^{-\alpha}$};
  \draw[color=black] plot[domain=0.45:2.2] (\x, 0.5/\x^1.5) node[above right]{$v=u^{-\alpha}$};

  \draw[color=black] (-0.02,0.5*2^1.5) -- (2,0.5*2^1.5) -- (2,-0.02);
  \draw[color=black] (-0.02,0.5*0.5^1.5) -- (0.5,0.5*0.5^1.5) -- (0.5,-0.02);
  \draw[color=black] (-0.02,0.5*2^1.5/3^1.5) -- (1.5,0.5*2^1.5/3^1.5);
  \draw[color=black] (2/3,0.5*1.5^1.5) -- (2/3,-0.02);
  \draw[color=black] (-0.02,0.5*1.5^1.5) -- (1.5,0.5*1.5^1.5) -- (1.5,-0.02);

  \draw[color=black,dashed] (0.5,0.5*2^1.5) -- (0.5,0.5*0.5^1.5) -- (2,0.5*0.5^1.5);
  \draw[color=black,dashed] (1,-0.02) -- (1,0.5);

  \draw (0.5,-0.02) node[below]{$s$};
  \draw (2/3,-0.02) node[below]{$r$};
  \draw (1,-0.02) node[below]{$\sqrt{sS}$};
  \draw (1.5,-0.02) node[below]{$sS/r$};
  \draw (2,-0.02) node[below]{$S$};

  \draw (-0.02,0.5*0.5^1.5) node[left]{$s^{-\delta}$};
  \draw (-0.02,0.5*2^1.5/3^1.5) node[left]{$(sS/r)^{-\alpha}$};
  \draw (-0.02,0.5*1.5^1.5) node[left]{$r^{-\alpha}$};
  \draw (-0.02,0.5*2^1.5) node[left]{$s^{-\alpha}$};
\end{tikzpicture}
\caption{A `node', its `anti-node', and their `leaves'} \label{fig:nodes_and_leaves}
\end{figure}
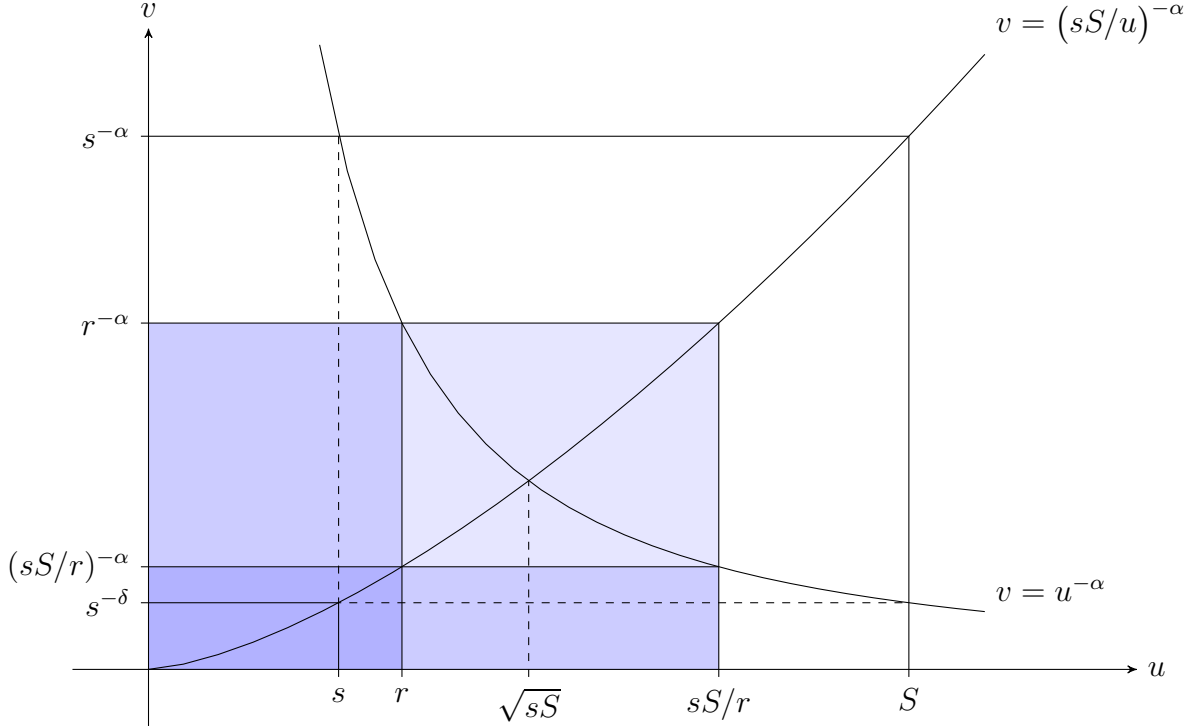

\begin{lemma}\label{l:there_is_a_good_leaf}
  Suppose $\alpha<\delta$.
  Let $\Sigma$ be an arbitrary discrete subset of $\R^d$ with the following two properties:

  \textup{(i)} every `leaf' in $\mathfrak L$ contains a point of $\Sigma$;

  \textup{(ii)} the root `node' $\cQ_s$ contains no points of $\Sigma$.

  Then there is a `leaf' containing two distinct points of $\Sigma$, one of which lies in the `node' that produces the `leaf'.
\end{lemma}

\begin{proof}
  Denote by $r_0$ the smallest $r$ such that the `node' $\cQ_r$ contains a point $\vec v$ of $\Sigma$.
  The existence of $r_0$ follows from the fact that $\cQ_{\sqrt{sS}}\in\mathfrak L$.
  Then $r_0>s$ and $\vec v$ lies on the boundary of $\cQ_{r_0}$.
  Since, by Proposition \ref{prop:leaves_production}, this `node' coincides with the intersection of its `leaves', $\vec v$ lies on the boundary of one of them, say, $\cQ(r_1,\alpha)$. Since $r_1$ equals either $r_0$, or $sS/r_0$, we have
  \[s<r_0\leq r_1\leq sS/r_0<S.\]
  If there are no other points of $\Sigma$ in $\cQ(r_1,\alpha)$, let us perturb this `leaf' by adding a small $\varepsilon$ to $r_1$, so that $\vec v$ is no longer in $\cQ(r_1+\varepsilon,\alpha)$. Since `leaves' are compact and $\Sigma$ is discrete, for $\varepsilon$ small enough no other points of $\Sigma$ will enter $\cQ(r_1+\varepsilon,\alpha)$. This contradicts property \textup{(i)}, which proves that along with $\vec v$ there is another point of $\Sigma$ in $\cQ(r_1,\alpha)$, distinct from $\vec v$.
\end{proof}

\begin{lemma}\label{l:good_node_and_antinode}
  Within the hypothesis of Lemma \ref{l:there_is_a_good_leaf} there are two distinct points of $\Sigma$, such that one of them lies in a `node' $\cQ_r$ and the other one lies in the corresponding `anti-node' $\cQ_{sS/r}$.
\end{lemma}

\begin{proof}
  Consider the `leaf' provided by Lemma \ref{l:there_is_a_good_leaf}. Then the `node' and the `anti-node' that produce this leaf satisfy the statement of the Lemma.
\end{proof}

Lemma \ref{l:good_node_and_antinode} is the key ingredient provided by the `stem'-and-`leaves' approach for proving Theorem \ref{t:weighted_German}. The only additional statement we need to formulate before we can proceed to the proof itself is the following technical Lemma.

\begin{lemma}\label{l:area_estimate}
  Suppose $\vec a,\vec b\in\R^2$ satisfy
  \[\vec a\in
    \Big\{ (z_1,z_2)\in\R^2 \,\Big|\, |z_1|\leq a,\ |z_2|\leq b \Big\},\quad
    \vec b\in
    \Big\{ (z_1,z_2)\in\R^2 \,\Big|\, |z_1|\leq A,\ |z_2|\leq B \Big\}.\]
  Then $|\vec a\wedge\vec b|\leq2\max(aB,bA)$.
\end{lemma}

The proof is elementary and we leave it to the reader.

\subsection{Proof of Theorem \ref{t:weighted_German}}

Let us keep on holding to the notation of Section \ref{sec:dual_lattices_and_parallelepipeds} and reformulate Definition \ref{def:uniform_weighted} the same way we reformulated Definition \ref{def:ordinary_weighted}.

\begin{proposition}\label{prop:uniform_omega_def_reformulation}
  \begin{equation*}
  \begin{aligned}
    \hat\omega_{\pmb\sigma,\pmb\rho}(\Theta) & =\sup\Big\{ \gamma\geq1 \,\Big|\, \text{for every $t$ large enough }\cP(t,\gamma)\cap\La\neq\{\vec 0\} \Big\}, \\
    \hat\omega_{\pmb\rho,\pmb\sigma}(\tr\Theta) & =\sup\Big\{ \alpha\geq1 \,\Big|\, \text{for every $r$ large enough }\cQ(r,\alpha)\cap\La^\ast\neq\{\vec 0\} \Big\}.
  \end{aligned}
  \end{equation*}
\end{proposition}

For every $\alpha\geq1$ set
\begin{equation}\label{eq:gamma_via_alpha}
  \gamma=
  \begin{cases}
    \dfrac{1-\rho_n\alpha^{-1}}
          {1-\rho_n}        \hskip 7mm\text{ if }\alpha\geq\rho_n/\sigma_m \\
    \hskip 2mm
    \dfrac{1-\sigma_m\vphantom{1^{\textstyle|}}}
          {1-\sigma_m\alpha}\hskip 8.2mm\text{ if }\alpha\leq\rho_n/\sigma_m
  \end{cases}.
\end{equation}
Then
\[1\leq\gamma<(1-\rho_n)^{-1}.\]
As in Section \ref{sec:twisted_Dyson_proof_itself}, for every $t>1$ set
\begin{equation}\label{eq:s_and_delta_via_t_and_gamma_uniform_case}
  s=t^{\gamma-(\gamma-1)\rho_n^{-1}},\qquad
  \delta=\frac{1+(\gamma-1)\sigma_m^{-1}}{\gamma-(\gamma-1)\rho_n^{-1}}.
\end{equation}
It is a simple exercise to show that with this choice of parameters we have
\[
  \alpha=1\iff\gamma=1\iff\delta=1
  \qquad\text{and}\qquad
  \alpha<\delta\iff\alpha>1.\]
We will prove Theorem \ref{t:weighted_German} by showing that
\begin{equation*}
  \hat\omega_{\pmb\rho,\pmb\sigma}(\tr\Theta)\geq\alpha
  \implies
  \hat\omega_{\pmb\sigma,\pmb\rho}(\Theta)\geq\gamma.
\end{equation*}
If $\hat\omega_{\pmb\rho,\pmb\sigma}(\tr\Theta)=1$, we are to take $\alpha=\gamma=1$, which makes both inequalities trivial. So we may assume hereafter that
\begin{equation}\label{eq:agreement}
  \hat\omega_{\pmb\rho,\pmb\sigma}(\tr\Theta)>1.
\end{equation}
Let S, $\mathfrak S$, $\mathfrak A$, $\mathfrak L$ be as in Section \ref{sec:nodes_and_leaves}.

\begin{lemma}\label{l:main}
  Suppose $\alpha>1$. Suppose that

  \textup{(i)} every `leaf' in $\mathfrak L$ contains a nonzero point of $\La^\ast$;

  \textup{(ii)} the root `node' $\cQ_s=\cQ(s,\delta)$ contains no nonzero points of $\La^\ast$.

  Then
  \[(2\cP(t,\gamma)^{\circledast})\cap\La^{\circledast}\neq\{\vec 0\}.\]
\end{lemma}

\begin{proof}
  By Definition \ref{def:second_pseudo_compound}
  \begin{equation*}
    \cP(t,\gamma)^\circledast=\left\{\,\vec Z 
                                       \in{\textstyle\bigwedge^2\R^d} \ \middle|
                                       \begin{array}{lr}
                                         |Z_{ij}|\leq t^{-\sigma_i-\sigma_j+1-\gamma}, & 1\leq i<j\leq m \\
                                         |Z_{m+i\,m+j}|\leq t^{\gamma\rho_i+\gamma\rho_j+1-\gamma}, & 1\leq i<j\leq n \\
                                         |Z_{j\,m+i}|\leq t^{-\sigma_j+\gamma\rho_i+1-\gamma}, & 1\leq j\leq m,\ 1\leq i\leq n
                                       \end{array} \right\}
  \end{equation*}
  with no first or second line of inequalities if $m=1$ or $n=1$ respectively.

  Since $\alpha>1$, we have $\alpha<\delta$. Let us apply Lemma \ref{l:good_node_and_antinode} with $\Sigma=\La^\ast\backslash\{\vec 0\}$. Then there are two distinct nonzero points $\vec v_1,\vec v_2\in\La^\ast$ and an $r\in\R$, $s<r<S$, such that
  \begin{equation*}
  \begin{aligned}
    \vec v_1\in\cQ_r & =
    \Bigg\{\,\vec z\in\R^d \ \Bigg|
             \begin{array}{lr}
               |z_j|\leq (sS/r)^{-\alpha\sigma_j}, & j=1,\ldots,m \\
               |z_{m+i}|\leq r^{\rho_i}, & i=1,\ldots,n
             \end{array} \Bigg\}, \\
    \vec v_2\in\cQ_{sS/r} & =
    \Bigg\{\,\vec z\in\R^d \ \Bigg|
             \begin{array}{lr}
               |z_j|\leq r^{-\alpha\sigma_j}, & j=1,\ldots,m \\
               |z_{m+i}|\leq (sS/r)^{\rho_i}, & i=1,\ldots,n
             \end{array} \Bigg\}.
  \end{aligned}
  \end{equation*}
  Let us show that
  \begin{equation*}
    \vec v_1\wedge\vec v_2\in2\cP(t,\gamma)^{\circledast}.
  \end{equation*}
  This will prove the Lemma.

  We are to show that the coefficients in the representation
  \[\vec v_1\wedge\vec v_2=\sum_{1\leq i<j\leq d}V_{ij}\vec e_i\wedge\vec e_j\]
  satisfy
  \begin{align}
    & |V_{ij}|\cdot t^{\sigma_i+\sigma_j-1+\gamma}\leq2, & 1\leq i<j\leq m, \label{eq:sigma_sigma} \\
    & |V_{m+i\,m+j}|\cdot t^{-\gamma\rho_i-\gamma\rho_j-1+\gamma}\leq2, & 1\leq i<j\leq n, \label{eq:rho_rho} \\
    & |V_{j\,m+i}|\cdot t^{\sigma_j-\gamma\rho_i-1+\gamma}\leq2, & 1\leq j\leq m,\ 1\leq i\leq n. \label{eq:sigma_rho}
  \end{align}
  We shall make use of the inequalities
  \begin{equation}\label{eq:negative}
  \begin{aligned}
    \gamma(1-\rho_n)-(1-\rho_n\alpha^{-1})\leq0, \\
    \gamma(1-\sigma_m\alpha)-(1-\sigma_m)\leq0,
  \end{aligned}
  \end{equation}
  that, as follows from \eqref{eq:gamma_via_alpha}, hold for every $\alpha\geq1$.

  \paragraph{Checking \eqref{eq:sigma_sigma}}

  By Lemma \ref{l:area_estimate} for $V_{ij}$ with $1\leq i<j\leq m$ we have
  \[
  \begin{aligned}
    |V_{ij}| & \leq
    2\max(r^{-\alpha\sigma_i}(sS/r)^{-\alpha\sigma_j},r^{-\alpha\sigma_j}(sS/r)^{-\alpha\sigma_i})\leq \\ & \leq
    2\max(s^{-\alpha\sigma_i}S^{-\alpha\sigma_j},s^{-\alpha\sigma_j}S^{-\alpha\sigma_i})= 
    2\max(s^{-\alpha\sigma_i-\delta\sigma_j},s^{-\alpha\sigma_j-\delta\sigma_i}).
  \end{aligned}
  \]
  It follows from \eqref{eq:s_and_delta_via_t_and_gamma_uniform_case} and \eqref{eq:negative} that
  \[
  \begin{aligned}
    s^{-\alpha\sigma_i-\delta\sigma_j}t^{\sigma_i+\sigma_j-1+\gamma} & =
    t^{-\left(\gamma-(\gamma-1)\rho_n^{-1}\right)\alpha\sigma_i-\left(1+(\gamma-1)\sigma_m^{-1}\right)\sigma_j}
    t^{\sigma_i+\sigma_j-1+\gamma}= \\ & =
    t^{\sigma_i-\left(\gamma-(\gamma-1)\rho_n^{-1}\right)\alpha\sigma_i+(\gamma-1)\left(1-\sigma_m^{-1}\sigma_j\right)}\leq \\ & \leq
    t^{\sigma_i-\left(\gamma-(\gamma-1)\rho_n^{-1}\right)\alpha\sigma_i}=
    t^{\alpha\sigma_i\rho_n^{-1}\left(\gamma(1-\rho_n)-(1-\rho_n\alpha^{-1})\right)}\leq1.
  \end{aligned}
  \]
  Similarly, interchanging $i$ and $j$, we get $s^{-\alpha\sigma_j-\delta\sigma_i}t^{\sigma_i+\sigma_j-1+\gamma}\leq1$.
  Thus, \eqref{eq:sigma_sigma} is fulfilled.

  \paragraph{Checking \eqref{eq:rho_rho}}

  By Lemma \ref{l:area_estimate} for $V_{m+i\,m+j}$ with $1\leq i<j\leq n$ we have
  \[
  \begin{aligned}
    |V_{m+i\,m+j}| & \leq
    2\max(r^{\rho_i}(sS/r)^{\rho_j},r^{\rho_j}(sS/r)^{\rho_i})\leq \\ & \leq
    2\max(s^{\rho_i}S^{\rho_j},s^{\rho_j}S^{\rho_i})= 
    2\max(s^{\rho_i+(\delta/\alpha)\rho_j},s^{\rho_j+(\delta/\alpha)\rho_i}).
  \end{aligned}
  \]
  It follows from \eqref{eq:s_and_delta_via_t_and_gamma_uniform_case} and \eqref{eq:negative} that
  \[
  \begin{aligned}
    s^{\rho_i+(\delta/\alpha)\rho_j}t^{-\gamma\rho_i-\gamma\rho_j-1+\gamma} & =
    t^{\left(\gamma-(\gamma-1)\rho_n^{-1}\right)\rho_i+\left(1+(\gamma-1)\sigma_m^{-1}\right)\alpha^{-1}\rho_j}
    t^{-\gamma\rho_i-\gamma\rho_j-1+\gamma}= \\ & =
    t^{-\gamma\rho_i+\left(1+(\gamma-1)\sigma_m^{-1}\right)\alpha^{-1}\rho_i+(\gamma-1)\left(1-\rho_n^{-1}\rho_j\right)}\leq \\ & \leq
    t^{-\gamma\rho_i+\left(1+(\gamma-1)\sigma_m^{-1}\right)\alpha^{-1}\rho_i}=
    t^{\alpha^{-1}\sigma_m^{-1}\rho_i\left(\gamma(1-\sigma_m\alpha)-(1-\sigma_m)\vphantom{1^{1}}\right)}\leq1.
  \end{aligned}
  \]
  Similarly, interchanging $i$ and $j$, we get $s^{\rho_j+(\delta/\alpha)\rho_i}t^{-\gamma\rho_i-\gamma\rho_j-1+\gamma}\leq1$. Thus, \eqref{eq:rho_rho} is fulfilled.

  \paragraph{Checking \eqref{eq:sigma_rho}}

  By Lemma \ref{l:area_estimate} for $V_{j\,m+i}$ with $1\leq j\leq m,\ 1\leq i\leq n$ we have
  \[
  \begin{aligned}
    |V_{j\,m+i}| & \leq
    2\max(r^{-\alpha\sigma_j+\rho_i},(sS/r)^{-\alpha\sigma_j+\rho_i})\leq \\ & \leq
    2\max(s^{-\alpha\sigma_j+\rho_i},S^{-\alpha\sigma_j+\rho_i})= 
    2\max(s^{-\alpha\sigma_j+\rho_i},s^{-\delta\sigma_j+(\delta/\alpha)\rho_i}).
  \end{aligned}
  \]
  It follows from \eqref{eq:s_and_delta_via_t_and_gamma_uniform_case} and \eqref{eq:negative} that
  \[
  \begin{aligned}
    s^{-\alpha\sigma_j+\rho_i}t^{\sigma_j-\gamma\rho_i-1+\gamma} & =
    t^{\left(\gamma-(\gamma-1)\rho_n^{-1}\right)\left(-\alpha\sigma_j+\rho_i\vphantom{1^{1}}\right)}t^{\sigma_j-\gamma\rho_i-1+\gamma}= \\ & =
    t^{\sigma_j-\left(\gamma-(\gamma-1)\rho_n^{-1}\right)\alpha\sigma_j+(\gamma-1)\left(1-\rho_n^{-1}\rho_i\right)}\leq \\ & \leq
    t^{\sigma_j-\left(\gamma-(\gamma-1)\rho_n^{-1}\right)\alpha\sigma_j}=
    t^{\alpha\sigma_j\rho_n^{-1}\left(\gamma(1-\rho_n)-(1-\rho_n\alpha^{-1})\right)}\leq1
  \end{aligned}
  \]
and
  \[
  \begin{aligned}
    s^{-\delta\sigma_j+(\delta/\alpha)\rho_i}t^{\sigma_j-\gamma\rho_i-1+\gamma} & =
    t^{\left(1+(\gamma-1)\sigma_m^{-1}\right)\left(-\sigma_j+\alpha^{-1}\rho_i\right)}t^{\sigma_j-\gamma\rho_i-1+\gamma}= \\ & =
    t^{-\gamma\rho_i+\left(1+(\gamma-1)\sigma_m^{-1}\right)\alpha^{-1}\rho_i+(\gamma-1)\left(1-\sigma_m^{-1}\sigma_j\right)}\leq \\ & \leq
    t^{-\gamma\rho_i+\left(1+(\gamma-1)\sigma_m^{-1}\right)\alpha^{-1}\rho_i}=
    t^{\alpha^{-1}\sigma_m^{-1}\rho_i\left(\gamma(1-\sigma_m\alpha)-(1-\sigma_m)\vphantom{1^{1}}\right)}\leq1.
  \end{aligned}
  \]
  Thus, \eqref{eq:sigma_rho} is also fulfilled.

  Hence $\vec v_1\wedge\vec v_2\in2\cP(t,\gamma)^{\circledast}$, which proves the Lemma.
\end{proof}

Having Lemma \ref{l:main} and Theorem \ref{t:second_pseudo_compound}, we can prove Theorem \ref{t:weighted_German} in the blink of an eye.

As we showed in Section \ref{sec:twisted_Dyson_proof_itself},
\begin{equation}\label{eq:if_Q_is_nonempty_then_so_is_P_uniform_case}
  \cQ(s,\delta)\cap\La^\ast\neq\{\vec 0\}
  \implies
  c\cP(t,\gamma)\cap\La\neq\{\vec 0\}.
\end{equation}
This observation, Lemma \ref{l:main}, and Theorem \ref{t:second_pseudo_compound} give us the key relation
\begin{equation}\label{eq:if_all_the_Q_are_nonempty_then_so_is_P}
\begin{aligned}
  \cQ(r,\alpha)\cap\La^\ast\neq\{\vec 0\}\text{ for every } & r\in[s,S]
  \implies \\ & \implies
  c''\cP(t,\gamma)\cap\La\neq\{\vec 0\},
  \vphantom{1^{|^{|}}}
\end{aligned}
\end{equation}
where $c''=\big(2d(d-1)\big)^{\raisebox{1ex}{$\frac1{2(d-2)}$}}$. Indeed, if $\cQ(s,\delta)$ contains a nonzero point of $\La^\ast$, then we are done by \eqref{eq:if_Q_is_nonempty_then_so_is_P_uniform_case}, since $c<c''$. Otherwise, the hypothesis of Lemma \ref{l:main} is satisfied, whence $(2\cP(t,\gamma)^{\circledast})\cap\La^{\circledast}\neq\{\vec 0\}$. Since $2\cP(t,\gamma)^{\circledast}=(c'''\cP(t,\gamma))^{\circledast}$ with $c'''=2^{1/(d-2)}$, Theorem \ref{t:second_pseudo_compound} gives $c''\cP(t,\gamma)\cap\La\neq\{\vec 0\}$. This proves \eqref{eq:if_all_the_Q_are_nonempty_then_so_is_P}.

Assuming \eqref{eq:agreement}, let us choose an arbitrary $\alpha$ such that $1<\alpha<\hat\omega_{\pmb\rho,\pmb\sigma}(\tr\Theta)$. Then for every $s$ large enough $\cQ(s,\alpha)\cap\La^\ast\neq\{\vec 0\}$. As we already noticed in Section \ref{sec:twisted_Dyson_proof_itself}, $s$ and $t$ tend to $+\infty$ simultaneously. Thus, in view of Proposition \ref{prop:uniform_omega_def_reformulation} relation \eqref{eq:if_all_the_Q_are_nonempty_then_so_is_P} implies that
\begin{equation*}
  \hat\omega_{\pmb\rho,\pmb\sigma}(\tr\Theta)>\alpha
  \implies
  \hat\omega_{\pmb\sigma,\pmb\rho}(\Theta)\geq\gamma.
\end{equation*}
Since $\gamma$ continuously depends on $\alpha$, we get for every $\alpha\geq1$
\begin{equation*}
  \hat\omega_{\pmb\rho,\pmb\sigma}(\tr\Theta)\geq\alpha
  \implies
  \hat\omega_{\pmb\sigma,\pmb\rho}(\Theta)\geq\gamma.
\end{equation*}
Hence
\[\hat\omega_{\pmb\sigma,\pmb\rho}(\Theta)\geq
  \begin{cases}
    \dfrac{1-\rho_n\hat\omega_{\pmb\rho,\pmb\sigma}(\tr\Theta)^{-1}}
          {1-\rho_n}        \hskip 7mm\text{ if }\hat\omega_{\pmb\rho,\pmb\sigma}(\tr\Theta)\geq\rho_n/\sigma_m \\
    \hskip 2mm
    \dfrac{1-\sigma_m\vphantom{1^{\textstyle|}}}
          {1-\sigma_m\hat\omega_{\pmb\rho,\pmb\sigma}(\tr\Theta)}\hskip 8.2mm\text{ if }\hat\omega_{\pmb\rho,\pmb\sigma}(\tr\Theta)\leq\rho_n/\sigma_m
  \end{cases}.\]
Swapping $(\pmb\sigma,\pmb\rho,\Theta)$ for $(\pmb\rho,\pmb\sigma,\tr\Theta)$ gives \eqref{eq:weighted_German}. Theorem \ref{t:weighted_German} is proved.

\section{Variety of inequalities generalising Dyson's theorem}\label{sec:variety}

As we noticed in Section \ref{sec:Mahler_by_German}, all the Dyson-like transference theorems actually base upon Theorem \ref{t:Mahler_by_German}. In the weighted setting, with the notation of Sections \ref{sec:Mahler_by_German}, \ref{sec:dual_lattices_and_parallelepipeds}, we can describe the scheme of a possible proof rather generally by the following diagram:
\begin{equation}\label{eq:Dyson_diagram}
  \cQ(s,\delta)
  \subseteq
  \cP^\ast
  \dashrightarrow
  \cP
  \subseteq
  \cP(t,\gamma).
\end{equation}
This diagram means that we find an appropriate parallelepiped $\cP$, to which we can apply Theorem \ref{t:Mahler_by_German}, and then choose $t$, $\gamma$, $s$, $\delta$ providing the inclusions in \eqref{eq:Dyson_diagram}. Given such $\cP$, $t$, $\gamma$, $s$, $\delta$, we can claim that, if there is a nonzero point of $\La^\ast$ in $\cQ(s,\delta)$, then there is a nonzero point of $\La$ in $c\cP(t,\gamma)$. In our proof of Theorem \ref{t:weighted_Dyson} (see Section \ref{sec:Dyson_proof}) we chose $\cP=\cP(t,\gamma)$. However, generally one can try and choose another $\cP$.

Let us consider arbitrary $t,\gamma,s,\delta\in\R$ such that
\[t>1,\quad\gamma\geq1,\quad s>1,\quad1\leq\gamma<(1-\rho_n)^{-1},\]
and arbitrary positive $\lambda_1,\ldots,\lambda_d$ determining $\cP$ by
\[\cP=\Big\{ \vec z=(z_1,\ldots,z_d)\in\R^d \,\Big|\, |z_i|\leq\lambda_i,\ i=1,\ldots,d \Big\}.\]
Let us define $a$, $b$, $\mu_1,\ldots,\mu_m$, $\nu_1,\ldots,\nu_n$, and $\Delta$ by
\[
\begin{aligned}
  & s=t^a,\quad \delta=b/a, \\
  & \lambda_j=t^{\mu_j},\quad
    \lambda_{m+i}=t^{\nu_i},\quad
    j=1,\ldots,m,\ i=1,\ldots,n, \\
  & \Delta=\sum_{1\leq j\leq m}\mu_j+\sum_{1\leq j\leq n}\nu_i.
\end{aligned}
\]
Consider the diagonal matrix $T=\textup{diag}(\lambda_1,\ldots,\lambda_d)$. Then $T^{-1}\cP=\cB$ and $T\cP^\ast=t^\Delta\cB$, where
\[\cB=\Big\{ \vec z=(z_1,\ldots,z_d)\in\R^d \,\Big|\, |z_i|\leq1,\ i=1,\ldots,d \Big\}.\]
Hence the inclusions in \eqref{eq:Dyson_diagram} are equivalent to $t^{-\Delta}T\cQ(s,\delta)\subseteq\cB\subseteq T^{-1}\cP(t,\gamma)$. Or, more explicitly,
\begin{equation*}
  \begin{array}{rrl}
    t^{-\Delta+\mu_j-b\sigma_j}\leq1,\quad &
    t^{-\mu_j+\sigma_j}\geq1,\quad &
    j=1,\ldots,m, \\
    t^{-\Delta+\nu_i+a\rho_i}\leq1,\quad &
    t^{-\nu_i-\gamma\rho_i}\geq1,\quad &
    \,i=1,\ldots,n.
  \end{array}
\end{equation*}
Thus, the inclusions in \eqref{eq:Dyson_diagram} take place if and only if for every $i$ and $j$ we have
\begin{align}
  \mu_j & \leq\sigma_j, & \nu_i & \leq-\gamma\rho_i, \label{eq:mu_nu} \\
  b & \geq(\mu_j-\Delta)/\sigma_j, & a & \leq(\Delta-\nu_i)/\rho_i. \label{eq:b_a}
\end{align}
We are interested in $\delta=b/a$ to be as small as possible, so for every $\mu_1,\ldots,\mu_m$, $\nu_1,\ldots,\nu_n$ it is best to set
\[b=\max_{1\leq j\leq m}\frac{\mu_j-\Delta}{\sigma_j}\,,\qquad
  a=\min_{1\leq i\leq n}\frac{\Delta-\nu_i}{\rho_i}\,.\]
Furthermore, notice that under the condition \eqref{eq:mu_nu} we have $\Delta-\mu_j\leq-\sigma_j+1-\gamma$ and $\Delta-\nu_i\leq\gamma\rho_i+1-\gamma$, since $\sigma_1+\ldots+\sigma_m=\rho_1+\ldots+\rho_n=1$. Hence
\[b=\max_{1\leq j\leq m}\frac{\mu_j-\Delta}{\sigma_j}\geq
  \max_{1\leq j\leq m}\big(1+(\gamma-1)\sigma_j^{-1}\big)=
  1+(\gamma-1)\sigma_m^{-1},\]
\[a=\min_{1\leq i\leq n}\frac{\Delta-\nu_i}{\rho_i}\leq
  \min_{1\leq i\leq n}\big(\gamma-(\gamma-1)\rho_i^{-1}\big)=
  \gamma-(\gamma-1)\rho_n^{-1},\]
with the equalities attained if for every $i$ and $j$ we have $\mu_j=\sigma_j$ and $\nu_i=-\gamma\rho_i$, i.e. if $\cP=\cP(t,\gamma)$. These values of $b$ and $a$ provide us with $\delta$ we used in Section \ref{sec:twisted_Dyson_proof_itself}.

However, if the choice of $\mu_1,\ldots,\mu_m$, $\nu_1,\ldots,\nu_n$ is not optimal, for instance, if at least one of the $m+n$ inequalities in \eqref{eq:mu_nu} is strict, then either $b$ is bounded away from $1+(\gamma-1)\sigma_m^{-1}$, or $a$ is bounded away from $\gamma-(\gamma-1)\rho_n^{-1}$. Thus, if $\cP$ is chosen as a proper subset of $\cP(t,\gamma)$, then $\delta$ will be greater than $\big(1+(\gamma-1)\sigma_m^{-1}\big)\big/\big(\gamma-(\gamma-1)\rho_n^{-1}\big)$. This is the reason the generalisation of Dyson's inequality obtained in \cite{ghosh_marnat_Pisa_2019} happens to be weaker than the one provided by Theorem \ref{t:weighted_Dyson}.

We leave it to the reader to prove that the weakest possible inequality that can be obtained in such a way corresponds to $\cP$ chosen so that $\cP^\ast=\cQ(s,\delta)$.

We end up with a remark that for the `non-twisted' case no problem of this kind arises, as in that case all the inequalities in \eqref{eq:mu_nu}, \eqref{eq:b_a} can be turned into equalities, providing thus a very nice relation $\cQ(s,\delta)=\cP(t,\gamma)^\ast$.

\paragraph{Acknowledgements.}

The author is a Young Russian Mathematics award winner and would like to thank its sponsors and jury.

\end{document}